\documentclass[12pt, a4]{amsart}
\usepackage{amsmath,amssymb,amsthm,eucal,bm}
\usepackage{graphicx}
\usepackage{cite,mathrsfs}
\numberwithin{equation}{section}

\newtheorem{theorem}{Theorem}[section]
\newtheorem{lemma}[theorem]{Lemma}
\newtheorem{prop}[theorem]{Proposition}

\theoremstyle{definition}

\theoremstyle{remark}
\newtheorem{remark}[theorem]{Remark}

\begin{document}
\title{An $M$-function associated with Goldbach's problem}

\author{Kohji Matsumoto}
\address{K. Matsumoto: Graduate School of Mathematics, Nagoya University, Chikusa-\
ku, Nagoya 464-8602, Japan}
\email{kohjimat@math.nagoya-u.ac.jp}


\begin{abstract}
We prove the existence of the $M$-function, by which we can state the limit theorem
for the value-distribution of the main term in the asymptotic formula for the
summatory function of the Goldbach generating function.
\end{abstract}

\date{}
\maketitle
%
%
\section{The Goldbach generating function}

One of the most famous unsolved problems in number theory is Goldbach's conjecture,
which asserts that all even integer $\geq 6$ can be written as a sum of two odd primes.

Let 
$$
r_2(n)=\sum_{l+m=n}\Lambda(l)\Lambda(m),
$$
where $\Lambda(\cdot)$ denotes the von Mangoldt function.
This may be regarded as the Goldbach generating function.    In fact, 
Goldbach's conjecture would imply $r_2(n)>0$ for all even $n\geq 6$.
Hardy and Littlewood \cite{HaLi24} conjectured that $r_2(n)\sim nS_2(n)$ for even $n$ as 
$n\to\infty$, where
$$
S_2(n)=\prod_{p|n}\left(1+\frac{1}{p-1}\right)
\prod_{p\nmid n}\left(1-\frac{1}{(p-1)^2}\right)
$$
($p$ denotes the primes).
In view of this conjecture, it is interesting to evaluate the sum
$$
A_2(x)=\sum_{n\leq x}(r_2(n)-nS_2(n)) \qquad (x>0).
$$
It is known that the estimate $A_2(x)=O(x^{3/2+\varepsilon})$ (where, and in what follows, 
$\varepsilon$ is an arbitrarily small positive number) is equivalent to the Riemann hypothesis (RH) for the Riemann zeta-function $\zeta(s)$ (see Granville \cite{Gran07},
Bhowmik and Ruzsa \cite{BhRu18}, Bhowmik et al. \cite{BHMS}).

The  unconditional estimate $A_2(x)=O(x^2(\log x)^{-A})$ ($A>0$) was classically known.
In 1991,
Fujii published a series of papers \cite{Fuj91} \cite{Fuj91a} \cite{Fuj91b}, in which
he refined this classical estimate under the 
RH.   Fujii first proved $A_2(x)=O(x^{3/2})$ in \cite{Fuj91}, and then in
\cite{Fuj91a}, he gave the following asymptotic formula
\begin{align}\label{A-Psi}
A_2(x)=-4x^{3/2}\cdot \Re\Psi(x) +R(x),
\end{align}
where $R(x)$ is the error term, and 
\begin{align}
\Psi(x)=\sum_{\gamma>0}\frac{x^{i\gamma}}{(1/2+i\gamma)(3/2+i\gamma)}
=\sum_{m=1}^{\infty}\frac{x^{i\gamma_m}}{(1/2+i\gamma_m)(3/2+i\gamma_m)},
\end{align}
with $\gamma$ running over all imaginary parts of non-trivial zeros of $\zeta(s)$ which are
positive.    We number those imaginary parts as $0<\gamma_1<\gamma_2<\cdots<\gamma_m<\cdots$.

Concerning the error term $R(x)$, Fujii \cite{Fuj91a} showed the estimate
$R(x)=O(x^{4/3}(\log x)^{4/3})$.
Egami and the author \cite{EgMa07} raised the conjecture
$$
R(x)=O(x^{1+\varepsilon}), \qquad R(x)=\Omega(x).
$$
This conjecture was settled by
Bhowmik and Schlage-Puchta \cite{BhSc10} in the form
$$
R(x)=O(x(\log x)^5), \qquad R(x)=\Omega(x\log\log x).
$$
The best upper-bound estimate at present is $O(x(\log x)^3)$
(Languasco and Zaccagnini\cite{LZ}; see also Goldston and Yang \cite{GY}).
As for the more detailed history, see \cite{BhHa20}.

Properties of
the main term on the right-hand side of \eqref{A-Psi} was first considered by
Fujii \cite{Fuj91b}.
Let 
\begin{align}
f(\alpha)=\Psi(e^{\alpha}) \qquad (\alpha\in\mathbb{R}).
\end{align}
In \cite{Fuj91b}, Fujii studied the value-distribution of $f(\alpha)$, and proved
the following limit theorem.     Assume that $\gamma$'s are linearly independent over
$\mathbb{Q}$ (which we call the LIC).    Then Fujii stated the existence of the
``density function'' $F(x)$ ($z=x+iy\in\mathbb{C}$) for which
\begin{align}\label{Fujii-result}
\lim_{X\to\infty}\frac{1}{X}\mu\{0\leq\alpha\leq X\;|\; f(\alpha)\in R\}
=\iint_R F(x+iy)dxdy
\end{align}
holds for any rectangle $R$ in $\mathbb{C}$, where $\mu\{\cdot\}$ means the
one-dimensional Lebesgue measure.
This is an analogue of the following result of Bohr and Jessen \cite{BJ30} 
\cite{BJ32} for the
value-distribution of $\zeta(s)$.   Let $\sigma>1/2$.    Bohr and Jessen proved the
existence of a continuous function $F_{\sigma}(z)$ for which
\begin{align}\label{BJ-result}
\lim_{T\to\infty}\frac{1}{2T}\mu\{-T\leq t\leq T\;|\; \log\zeta(\sigma+it)\in R\}
=\iint_R F_{\sigma}(x+iy)dxdy
\end{align}
holds for any rectangle $R$.

Fujii gave a sketch of the proof, which is along the same line as in
\cite{BJ30}.
In particular, Fujii indicated explicitly how to construct 
$F(x+iy)$, following the method of Bohr and Jessen \cite{BJ29}.

In \cite{Fuj91b}, Fujii also studied the size of the oscillation of $\Psi(x)$.
This direction has recently been further pursued by Mossinghoff and Trudgian \cite{MT}.

\section{The theory of $M$-functions and the statement of the main result}

The result \eqref{BJ-result} of Bohr and Jessen has been generalized to a wider class
of zeta-functions.   The existence of the limit on the left-hand side of \eqref{BJ-result}
is now generalized to a fairly general class (see \cite{Mats90}).

It is more difficult to prove the integral expression like the right-hand side of
\eqref{BJ-result}.   The case of Dirichlet $L$-functions $L(s,\chi)$ is essentially the
same as in the case of $\zeta(s)$ (see Joyner \cite{Joyn86}).   The case of Dedekind
zeta-functions of algebraic number fields was studied by the author \cite{Mats91} \cite{Mat92} \cite{Mats07}.
The case of automorphic $L$-functions attached to ${\rm SL}(2,\mathbb{Z})$ or its 
congruence subgroups was established recently in \cite{MaUm19} \cite{MaUm20}.

All of those generalizations consider the situation when $t=\Im s$ varies
(like the left-hand side of \eqref{BJ-result}).   When we 
treat more general $L$-functions, various other aspects can be considered.
In 2008, Ihara \cite{Iha08} studied the $\chi$-aspect for $L$-functions defined on
number fields or function fields.
His study was then further refined in a series of papers of Ihara and the author
\cite{IM10} \cite{IM11} \cite{IM11a} \cite{IM14}.
Let us quote a result proved in \cite{IM11}.

\begin{theorem}\label{IM-th}
Let $s=\sigma+it\in\mathbb{C}$ with $\sigma>1/2$.    There exists an explicitly constructable density function $M_{\sigma}(w)$, continuous and non-negative, for
which
\begin{align}\label{IM-formula}
{\rm Avg}_{\chi}\Phi(\log L(s,\chi))=\int_{\mathbb{C}}M_{\sigma}(w)\Phi(w)|dw|
\end{align}
holds, where ${\rm Avg}_{\chi}$ stands for some average with respect to characters, 
$|dw|=dudv/(2\pi)$ {\rm (}for $w=u+iv${\rm )}, and $\Phi$ is the test function which is either
{\rm (i)} some continuous function, or {\rm (ii)} the characteristic function of
a compact subset of $\mathbb{C}$ or its complement. 
\end{theorem}

The density function $M_{\sigma}$ is called an $M$-function.
Here we do not give the details how to define ${\rm Avg}_{\chi}$, but in \cite{IM11},
two types of averages were considered.    One of them is a certain average with
respect to Dirichlet characters, and the other is essentially the same as the average
in $t$-aspect like \eqref{BJ-result}.   In this sense, $F_{\sigma}$ in \eqref{BJ-result}
may be regarded as an example of $M$-functions.

Since then, various analogues of Theorem \ref{IM-th} were discovered by
Mourtada and Murty \cite{MoMu15}, Akbary and Hamieh \cite{AkHa20},
Lebacque and Zykin \cite{LeZy18}, Matsumoto and Umegaki \cite{MaUm18}, 
Mine \cite{Mine19a} \cite{Mine19b} \cite{Mine20}, and so on.

The aim of the present article is to show the following ``limit theorem'', which is
a generalization of Fujii's 
\eqref{Fujii-result} in the framework of the theory of $M$-functions.

\begin{theorem}\label{thm-main}
We assume the LIC.
There exists an explicitly constructable density function {\rm (}$M$-function{\rm )} 
$M:\mathbb{C}\to\mathbb{R}_{\geq 0}$, for which
\begin{align}\label{main-formula}
\lim_{X\to\infty}\frac{1}{X}\int_0^X \Phi(f(\alpha))d\alpha=
\int_{\mathbb{C}}M(w)\Phi(w)|dw|
\end{align}
holds for any test function $\Phi:\mathbb{C}\to\mathbb{C}$ which is continuous, or 
which is the characteristic function of either a compact subset of $\mathbb{C}$ or
the complement of such a subset.   The function $M(w)$ is continuous, tends to $0$
when $|w|\to\infty$, $M(\overline{w})=M(w)$, and
\begin{align}
\int_{\mathbb{C}}M(w)|dw|=1.
\end{align}
\end{theorem}

\begin{remark}
Choosing $\Phi=\mathbf{1}_R$, we recover Fujii's result \eqref{Fujii-result}.
\end{remark}

The above theorem is an analogue of the absolutely convergent
case in the theory of $M$-functions (that is, an analogue of \cite[Theorem 4.2]{IM11}).
In this sense, our theorem is a rather simple example of $M$-functions.
In particular, complicated mean-value arguments (such as \cite[Sections 5--8]{IM11})
are not necessary.     Still, however, our theorem gives a new evidence of the
ubiquity of $M$-functions. 

\section{The finite truncation}

The rest of the present paper is devoted to the proof of Theorem \ref{thm-main}.

We first define the finite truncation of $f(\alpha)$.   
Let $b_m=(1/2+i\gamma_m)(3/2+i\gamma_m)$, $c_m=1/|b_m|$, and $\beta_m=\arg b_m$.   Then
\begin{align}
f(\alpha)=\sum_{m=1}^{\infty}\frac{e^{i\alpha\gamma_m}}{b_m}
=\sum_{m=1}^{\infty}c_m e^{i(\alpha\gamma_m-\beta_m)}.
\end{align}
It is to be noted that
\begin{align}\label{bound_c_m}
c_m=\frac{1}{\sqrt{\frac{1}{4}+\gamma_m^2}\sqrt{\frac{9}{4}+\gamma_m^2}}
\sim \frac{1}{\gamma_m^2}\sim \left(\frac{\log m}{2\pi m}\right)^2
\end{align}
as $m\to\infty$,
hence the above series expression of $f(\alpha)$ is absolutely convergent.

We first consider the finite truncation
\begin{align}
f_N(\alpha)=\sum_{m=1}^{N}c_m e^{i(\alpha\gamma_m-\beta_m)}.
\end{align}
Let $\mathbb{T}$ be the unit circle on $\mathbb{C}$, and
$\mathbb{T}_N=\prod_{m\leq N}\mathbb{T}$.   Define
\begin{align}\label{def_S_N}
S_N(\mathbf{t}_N)=\sum_{m\leq N}c_m t_m,
\end{align}
where $\mathbf{t}_N=(t_1,\ldots,t_N)\in\mathbb{T}_N$.   Then obviously
\begin{align}\label{f_N_S_N}
f_N(\alpha)=S_N(e^{i(\alpha\gamma_1-\beta_1)},\ldots,
e^{i(\alpha\gamma_N-\beta_N)}).
\end{align}
The idea of attaching the mapping 
$S_N:\mathbb{T}_N\to\mathbb{C}$ to $f_N$ goes back to the work of Bohr 
\cite{Bohr15}.    We denote by $d^*\mathbf{t}_N$ the normalized Haar measure
on $\mathbb{T}_N$, that is the product measure of $d^*t=(2\pi)^{-1}d\theta$ for 
$t=e^{i\theta}\in\mathbb{T}$.
The following is an analogue of \cite[Theorem 1]{Iha08}.

\begin{prop}\label{prop_N}
We may construct a function $M_N:\mathbb{C}\to\mathbb{R}_{\geq 0}$, for which
\begin{align}\label{formula_prop_N}
\int_{\mathbb{C}}M_N(w)\Phi(w)|dw|=\int_{\mathbb{T}_N}\Phi(S_N(\mathbf{t}_N))
d^*\mathbf{t}_N
\end{align}
holds for any continuous function $\Phi$ on $\mathbb{C}$.
In particular, choosing $\Phi\equiv 1$ we obtain
\begin{align}\label{total=1}
\int_{\mathbb{C}}M_N(w)|dw|=1.
\end{align}
Also for $N\geq 2$ the function $M_N(w)$ is compactly supported, 
non-negative and $M_N(\overline{w})=M_N(w)$.
\end{prop}

\begin{proof}
First consider the case $N=1$.
Let $s_n(t_n)=c_n t_n$.   For $w=re^{i\theta}\in\mathbb{C}$ ($r=|w|, \theta=\arg w$), 
define
\begin{align}
m_n(w)=\frac{1}{r}\delta(r-c_n),
\end{align}
where $\delta(\cdot)$ stands for the usual Dirac delta distribution.   We have
\begin{align}\label{m_n}
&\int_{\mathbb{C}}m_n(w)\Phi(w)|dw|\\
&=\int_0^{2\pi}\int_0^{\infty}m_n(re^{i\theta})\Phi(re^{i\theta})\frac{1}{2\pi}
rdrd\theta\notag\\
&=\frac{1}{2\pi}\int_0^{2\pi}d\theta\int_0^{\infty}\delta(r-c_n)
\Phi(re^{i\theta})dr\notag\\
&=\frac{1}{2\pi}\int_0^{2\pi}\Phi(c_n e^{i\theta})d\theta\notag\\
&=\int_{\mathbb{T}}\Phi(s_n (t_n))d^*t_n.\notag
\end{align}
In particular, putting $n=1$ in \eqref{m_n}, we find
\begin{align}
\int_{\mathbb{C}}m_1(w)\Phi(w)|dw|=\int_{\mathbb{T}}\Phi(s_1 (t_1))d^*t_1,
\end{align}
which implies that the case $N=1$ of Proposition \ref{prop_N} is valid with
$M_1=m_1$.

Now we prove the general case by induction on $N$.   Define
\begin{align}\label{def_M_N}
M_N(w)=\int_{\mathbb{C}}M_{N-1}(w')m_N(w-w')|dw'|
\end{align}
for $N\geq 2$.   This is compactly supported, and
\begin{align*}
&\int_{\mathbb{C}}M_N(w)\Phi(w)|dw|\\
&=\int_{\mathbb{C}}\int_{\mathbb{C}}M_{N-1}(w')m_N(w-w')|dw'|\Phi(w)|dw|\\
&=\int_{\mathbb{C}}M_{N-1}(w')|dw'|\int_{\mathbb{C}}m_N(w-w')\Phi(w)|dw|.
\end{align*}
The exchange of the integrations is verified because $M_N$ is compactly supported.
Putting $w''=w-w'$ we see that the inner integral is 
$$
=\int_{\mathbb{C}}m_N(w'')\Phi_{w'}(w'')|dw''| \qquad 
({\rm where} \quad\Phi_{w'}(w'')=\Phi(w''+w')),
$$
which is, by \eqref{m_n}, 
$$
=\int_{\mathbb{T}}\Phi_{w'}(s_N(t_N))d^* t_N.
$$
Therefore
\begin{align*}
&\int_{\mathbb{C}}M_N(w)\Phi(w)|dw|
=\int_{\mathbb{C}}M_{N-1}(w')|dw'|\int_{\mathbb{T}}\Phi_{w'}(s_N(t_N))d^* t_N\\
&\quad=\int_{\mathbb{T}}d^*t_N\int_{\mathbb{C}}M_{N-1}(w')\Phi_{w'}(s_N(t_N))|dw'|\\
&\quad=\int_{\mathbb{T}}d^*t_N\int_{\mathbb{C}}M_{N-1}(w')\Phi_{s_N}(w')|dw'|,
\end{align*}
where $\Phi_{s_N}(w')=\Phi(s_N(t_N)+w')=\Phi_{w'}(s_N(t_N))$.    Using the induction assumption
we see that the right-hand side is
$$
=\int_{\mathbb{T}}d^*t_N\int_{\mathbb{T}_{N-1}}\Phi_{s_N}(S_{N-1}(\mathbf{t}_{N-1}))
d^*\mathbf{t}_{N-1}
=\int_{\mathbb{T}_N}\Phi_{s_N}(S_{N-1}(\mathbf{t}_{N-1}))
d^*\mathbf{t}_{N}.
$$
Since
$$
\Phi_{s_N}(S_{N-1}(\mathbf{t}_{N-1}))=\Phi(S_{N-1}(\mathbf{t}_{N-1})+s_N(t_N))
=\Phi(S_N(\mathbf{t}_N)),
$$
we obtain the assertion of the proposition.
\end{proof}

The following two propositions are analogues of
\cite[Remark 3.2 and Remark 3.3]{IM11}.    For any $A\subset\mathbb{C}$, by $\mathbf{1}_A$
we denote the characteristic function of $A$.
By ${\rm Supp}(\phi)$ we mean the support of a function $\phi$.

\begin{prop}\label{prop_char_fct}
The formula \eqref{formula_prop_N} is valid when $\Phi=\mathbf{1}_A$, 
where $A$ is either
a compact subset of $\mathbb{C}$ or the complement of such a subset.
\end{prop}

\begin{proof}
It is enough to prove the case when $A$ is compact.   
Let $\phi_1,\phi_2$ be continuous non-negative functions, defined on $\mathbb{C}$,
compactly supported, satisfying $0\leq \phi_1\leq\mathbf{1}_A\leq\phi_2\leq 1$ and
${\rm Vol}({\rm Supp}(\phi_2-\phi_1))<\varepsilon$ (where ``Vol'' denotes the volume
measured by $d^*\mathbf{t}_{N}$).   Then
$$
\int_{\mathbb{C}}M_N(w)(\mathbf{1}_A-\phi_1)(w)|dw|<C_N\varepsilon,\;
\int_{\mathbb{C}}M_N(w)(\phi_2-\mathbf{1}_A)(w)|dw|<C_N\varepsilon,
$$
where $C_N=\sup\{M_N(w)\}$.   Therefore, using Proposition \ref{prop_N} we have
\begin{align*}
&\int_{\mathbb{C}} M_N(w)\mathbf{1}_A(w)|dw|-C_N\varepsilon
\leq \int_{\mathbb{C}} M_N(w)\phi_1(w)|dw|\\
&=\int_{\mathbb{T}_N}\phi_1(S_N(\mathbf{t}_N))d^*\mathbf{t}_N
\leq \int_{\mathbb{T}_N}\mathbf{1}_A(S_N(\mathbf{t}_N))d^*\mathbf{t}_N\\
&\leq \int_{\mathbb{T}_N}\phi_2(S_N(\mathbf{t}_N))d^*\mathbf{t}_N
=\int_{\mathbb{C}} M_N(w)\phi_2(w)|dw|\\
&\leq \int_{\mathbb{C}} M_N(w)\mathbf{1}_A(w)|dw|+C_N\varepsilon,
\end{align*}
from which the desired assertion follows.
\end{proof}

In the proof of Proposition \ref{prop_N} we have shown that $M_N$ is compactly supported.
Now we show more explicitly 
what is the support.    

\begin{prop}\label{prop_support}
The support of $M_N$ is the image of the mapping $S_N$.
\end{prop}

\begin{proof}
Let $A$ be a compact subset of $\mathbb{C}$.    We can use \eqref{formula_prop_N} with
$\Phi=\mathbf{1}_A$ because of Proposition \ref{prop_char_fct}.    Then
\begin{align}
\int_A M_N(w)|dw|=\int_{\mathbb{T}_N}\mathbf{1}_A(S_N(\mathbf{t}_N))d^*\mathbf{t}_N
={\rm Vol}(S_N^{-1}(A)),
\end{align}
which implies the proposition.
\end{proof}

\section{The finite-truncation version of the theorem}

The aim of this section is to prove

\begin{prop}\label{prop_lim}
Under the assumption of the LIC, we have
\begin{align}
\lim_{X\to\infty}\frac{1}{X}\int_0^X \Phi(f_N(\alpha))d\alpha=
\int_{\mathbb{T}_N}\Phi(S_N(\mathbf{t}_N))
d^*\mathbf{t}_N
\end{align}
for any continuous function $\Phi$ on $\mathbb{C}$.
\end{prop}
Then, combining this with Proposition \ref{prop_N}, we have
\begin{align}\label{finite-main-formula}
\lim_{X\to\infty}\frac{1}{X}\int_0^X \Phi(f_N(\alpha))d\alpha=
\int_{\mathbb{C}}M_N(w)\Phi(w)|dw|
\end{align}
for any continuous $\Phi$,
which is the ``finite-truncation'' analogue of our main theorem.

In view of \eqref{f_N_S_N}, 
in order to prove Proposition \ref{prop_lim}, it is enough to prove the following

\begin{prop}\label{prop_lim_2}
Under the assumption of the LIC, we have
\begin{align}
\lim_{X\to\infty}\frac{1}{X}\int_0^X \Psi(e^{i(\alpha\gamma_1-\beta_1)}
,\ldots,e^{i(\alpha\gamma_N-\beta_N)})d\alpha=
\int_{\mathbb{T}_N}\Psi(\mathbf{t}_N)d^*\mathbf{t}_N
\end{align}
holds for any continuous $\Psi:\mathbb{T}_N\to\mathbb{C}$.
\end{prop}

This is an analogue of \cite[Lemma 4.3.1]{Iha08}.

\begin{proof}
Write $\mathbf{t}_N=(e^{ i\theta_1},\ldots,e^{ i\theta_N})$.
Then the right-hand side of Proposition \ref{prop_lim_2} is
$$
= \frac{1}{(2\pi)^N}\int_0^{2\pi}\cdots\int_0^{2\pi}
\Psi(e^{ i\theta_1},\ldots,e^{ i\theta_N})
d\theta_1\cdots d\theta_N.
$$
To show that this is equal to the left-hand side, by Weyl's criterion (see 
\cite[Chapter 1, Theorem 9.9]{KN}), it is enough to show the equality when
$\Psi=t_1^{n_1}\cdots t_N^{n_N}$ for any
$(n_1,\ldots,n_N)\in\mathbb{Z}^N\setminus\{(0,\ldots,0)\}$.
But in this case, since
$\Psi(e^{ i\theta_1},\ldots,e^{ i\theta_N})=e^{i(n_1\theta_1+\cdots+n_N\theta_N)}$,
the right-hand side is clearly equal to $0$.
The left-hand side is
\begin{align*}
&=\lim_{X\to\infty}\frac{1}{X}\int_0^X
e^{in_1(\alpha\gamma_1-\beta_1)+\cdots+in_N(\alpha\gamma_N-\beta_N)}d\alpha\\
&=\lim_{X\to\infty}\frac{1}{X}e^{-i(n_1\beta_1+\cdots+n_N\beta_N)}
\int_0^X e^{i\alpha(n_1\gamma_1+\cdots+n_N\gamma_N)}d\alpha.
\end{align*}
Since we assume the LIC, $n_1\gamma_1+\cdots+n_N\gamma_N\neq 0$ because
$(n_1,\ldots,n_N)\neq(0,\ldots,0)$.   Therefore the above is
$$
=\lim_{X\to\infty}\frac{1}{X}e^{-i(n_1\beta_1+\cdots+n_N\beta_N)}
\cdot\frac{e^{iX(n_1\beta_1+\cdots+n_N\beta_N)}-1}{i(n_1\beta_1+\cdots+n_N\beta_N)}
$$
which is also equal to $0$.   The proposition is proved.
\end{proof}

\section{The existence of the $M$-function}

In this section we prove the existence of the limit function
\begin{align}
M(w)=\lim_{N\to\infty}M_N(w).
\end{align}

For this purpose we consider the Fourier transform.   We follow the argument
on pp.644-647 in \cite{IM11}, which is based on the ideas of Ihara \cite{Iha08}
and of the author \cite{Mat92}.

Let $\psi_z(w)=\exp(i\Re (\overline{z} w))$, and define the Fourier transform of
$m_n$ as
\begin{align}
\widetilde{m}_n(z)=\int_{\mathbb{C}}m_n(w)\psi_z(w)|dw|.
\end{align}
Applying \eqref{m_n} with $\Phi=\psi_z$, we see that the right-hand side of the above is
\begin{align*}
&=\int_{\mathbb{T}}\psi_z(s_n(t_n))d^*t_n
=\frac{1}{2\pi}\int_0^{2\pi}\psi_z(c_n e^{i\theta_n})d\theta_n\\
&=\frac{1}{2\pi}\int_0^{2\pi}\exp(i\Re(\overline{z}\cdot c_ne^{i\theta_n}))d\theta_n.
\end{align*}
Writing
$\overline{z}\cdot c_ne^{i\theta_n}=c_n |z|e^{i(\theta_n-\tau)}$ ($\tau=\arg z$), 
we have
\begin{align}\label{inRHS}
\Re(\overline{z}\cdot c_ne^{i\theta_n})
=c_n|z|\cos(\theta_n-\tau)
=c_n|z|(\cos\theta_n\cos\tau+\sin\theta_n\sin\tau)
\end{align}
and so
\begin{align}\label{RHS}
\widetilde{m}_n(z)=\frac{1}{2\pi}\int_0^{2\pi}\exp(ic_n|z|(\cos\theta_n\cos\tau
+\sin\theta_n\sin\tau)d\theta_n. 
\end{align}
Now quote:

\begin{lemma}
{\rm (Jessen and Wintner \cite[Theorem 12]{JW35})}
Let $C$ be a closed convex curve in $\mathbb{C}$
parametrized by $x(\theta)=(\xi_1(\theta),\xi_2(\theta))$, 
$z=|z|e^{i\tau}\in\mathbb{C}$, and let
$g_{\tau}(\theta)=\xi_1(\theta)\cos\tau+\xi_2(\theta)\sin\tau$.
Assume that $\xi_1,\xi_2\in C^2$ and $g_{\tau}^{\prime\prime}(\theta)$ has (for each
fixed $\tau$) exactly two zeros on $C$.   Then
\begin{align}
\int_C \exp(i|z|g_{\tau}(\theta))d\theta = O(|z|^{-1/2}),
\end{align}
where the implied constant depends on $C$.
\end{lemma}

In the present case $\xi_1(\theta)=c_n \cos\theta$,
$x_2(\theta)=c_n\sin\theta$, and $C$ is the circle of radius $c_n$.   Since
$$
g_{\tau}^{\prime\prime}(\theta)
=-c_n(\cos\theta\cos\tau+\sin\theta\sin\tau)=-c_n\cos(\theta-\tau),
$$
the assumption of the lemma is clearly satisfied, and hence by the lemma we have
\begin{align}\label{est_wtm_n}
\widetilde{m}_n(z)=O_n(|z|^{-1/2}).
\end{align}
Now define
\begin{align}\label{def_M_N_tilde}
\widetilde{M}_N(z)=\prod_{n\leq N}\widetilde{m}_n(z).
\end{align}
Then from \eqref{est_wtm_n} and the obvious inequality $|\widetilde{m}_n(z)|\leq 1$
(which immediately follows from \eqref{RHS}), we have
\begin{align}\label{bound-by-N}
\widetilde{M}_N(z)=O_N(|z|^{-N/2})
\end{align}
and
\begin{align}\label{bound-by-1}
|\widetilde{M}_N(z)|\leq 1.
\end{align}
From these inequalities we obtain (i) and (ii) of the following

\begin{prop}\label{prop-4}
Let $N_0\geq 5$.

{\rm (i)} $\widetilde{M}_{N_0}\in L^t$ for any $t\in[1,+\infty]$,

{\rm (ii)} $|\widetilde{M}_N(z)|\leq |\widetilde{M}_{N_0}(z)|$ for all $N\geq N_0$,

{\rm (iii)} $\widetilde{M}_N(z)$ converges to a certain function $\widetilde{M}(z)$
uniformly in any compact subset when $N\to\infty$.
\end{prop}

\begin{proof}[Proof of {\rm (iii)}]
It is clear from \eqref{inRHS} that
\begin{align}
\frac{1}{2\pi}\int_0^{2\pi}\Re(\overline{z}\cdot c_ne^{i\theta_n})d\theta_n =0.
\end{align}
Therefore we can write
\begin{align}
\widetilde{m}_n(z)-1=\frac{1}{2\pi}\int_0^{2\pi}
(\exp(i\Re(\overline{z}\cdot c_ne^{i\theta_n}))-1-\Re(\overline{z}\cdot c_ne^{i\theta_n}))
d\theta_n.
\end{align}
Since $|e^{ix}-1-ix|\ll x^2$ for any real $x$ (by the Taylor expansion for small $|x|$, 
and by the fact $|e^{ix}|=1$ for large $|x|$), we obtain
\begin{align}\label{bound_m_n}
|\widetilde{m}_n(z)-1|
\ll \int_0^{2\pi}|\Re(\overline{z}\cdot c_ne^{i\theta_n})|^2d\theta_n
\ll |z|^2 c_n^2.
\end{align}
Let $N<N'$.   Then
\begin{align*}
|\widetilde{M}_{N'}(z)-\widetilde{M}_N(z)|
&\leq \sum_{j=1}^{N'-N}|\widetilde{M}_{N+j}(z)-\widetilde{M}_{N+j-1}(z)|\\
&=\sum_{j-1}^{N'-N}|\widetilde{M}_{N+j-1}(z)|\cdot |\widetilde{m}_{N+j}(z)-1|\\
&\ll |z|^2 \sum_{j=1}^{N'-N}c_{N+j}^2
\end{align*}
by \eqref{bound-by-1} and \eqref{bound_m_n}.     Because of \eqref{bound_c_m} we see
that the series on the right-hand side converges as $N,N'\to\infty$.
Therefore by Cauchy's criterion we obtain the assertion (iii).
\end{proof}

Now we prove the following result, which is an analogue of
\cite[Proposition 3.4]{IM11}.

\begin{prop}
$\widetilde{M}_N(z)$ converges to $\widetilde{M}(z)$
uniformly in $\mathbb{C}$ when $N\to\infty$.
The limit function $\widetilde{M}(z)$ is continuous and belongs to $L^t$
{\rm (}for any $t\in [1,\infty]${\rm )}, and the above convergence is also 
$L^t$-convergence.
\end{prop}

\begin{proof}
Let $0<\varepsilon<1$.     By Proposition \ref{prop-4} (i) we can find 
$R=R(N_0)>1$ for which
\begin{align}\label{ineq1}
\int_{|z|\geq R}|\widetilde{M}_{N_0}(z)|^t|dz|<\varepsilon
\end{align}
for any $1\leq t<\infty$ and (noting \eqref{bound-by-N})
\begin{align}\label{ineq2}
\sup_{|z|\geq R}|\widetilde{M}_{N_0}(z)|<\varepsilon.
\end{align}
(Here $R$ is independent of $t$, because by \eqref{bound-by-1} the inequality 
\eqref{ineq1} for $t=1$ implies \eqref{ineq1} for other finite values of $t$.)
Because of Proposition \ref{prop-4} (ii), the above inequalities are valid also for
$\widetilde{M}_{N}(z)$ for all $N\geq N_0$.

Taking $N\to\infty$ in the above inequalities, we find that
$\widetilde{M}\in L^t$ ($1\leq t\leq \infty$).

Let $N'>N$.    Then
\begin{align*}
|\widetilde{M}_{N'}(z)-\widetilde{M}_N(z)|^t
=\left|\prod_{N<n\leq N'}m_n(z)-1\right|^t\cdot|\widetilde{M}_N(z)|^t
\leq 2^t |\widetilde{M}_N(z)|^t
\end{align*}
for any $z\in\mathbb{C}$, so taking the limit $N'\to\infty$ we have
\begin{align}\label{ineq3}
|\widetilde{M}(z)-\widetilde{M}_N(z)|^t
\leq 2^t |\widetilde{M}_N(z)|^t.
\end{align}
Therefore from \eqref{ineq1} and \eqref{ineq2} we obtain
\begin{align}
\int_{|z|\geq R}|\widetilde{M}(z)-\widetilde{M}_{N}(z)|^t|dz|<2^t \varepsilon
\end{align}
and (using the case $t=1$ of \eqref{ineq3})
\begin{align}
\sup_{|z|\geq R}|\widetilde{M}(z)-\widetilde{M}_{N}(z)|<2 \varepsilon
\end{align}
for all $N\geq N_0$.

Now we apply Proposition \ref{prop-4} (iii) for the compact subset $\{|z|\leq R\}$
to obtain that if $N=N(R,\varepsilon)\geq N_0$ is sufficiently large, then
\begin{align}
|\widetilde{M}(z)-\widetilde{M}_{N}(z)|\leq \varepsilon/R^2
\end{align}
for all $z$ satisfying $|z|\leq R$.   Therefore
\begin{align}
\int_{|z|\leq R}|\widetilde{M}(z)-\widetilde{M}_{N}(z)|^t|dz|< 
\pi R^2 \left(\frac{\varepsilon}{R^2}\right)^t \leq \pi R^2\frac{\varepsilon}{R^2}
\leq \pi\varepsilon
\end{align}
and 
\begin{align}
\sup_{|z|\leq R}|\widetilde{M}(z)-\widetilde{M}_{N}(z)|< \frac{\varepsilon}{R^2}
\leq \varepsilon.
\end{align}
Now we arrive at
\begin{align}
\int_{\mathbb{C}}|\widetilde{M}(z)-\widetilde{M}_{N}(z)|^t|dz|<(2^t+\pi) \varepsilon
\end{align}
and 
\begin{align}
\sup_{z\in \mathbb{C}}|\widetilde{M}(z)-\widetilde{M}_{N}(z)|<3 \varepsilon.
\end{align}
Therefore we obatin the assertions of the proposition.
\end{proof}

Since $M_N$ is given by the convolution product of $m_1,\ldots,m_N$ (see \eqref{def_M_N}),
by the definition \eqref{def_M_N_tilde}, $\widetilde{M}_N(z)$ is the 
Fourier transform of $M_N(w)$.
Therefore we can write
\begin{align}
M_N(w)=\int_{\mathbb{C}}\widetilde{M}_N(z)\psi_{-w}(z)|dz|.
\end{align}
Define
\begin{align}
M(w)=\int_{\mathbb{C}}\widetilde{M}(z)\psi_{-w}(z)|dz|.
\end{align}
Then we obtain

\begin{prop}\label{prop-6}
When $N\to\infty$, $M_N(w)$ converges to $M(w)$ uniformly in $w\in\mathbb{C}$.
The limit function $M(w)$ is continuous, non-negative, tends to $0$ when $|w|\to\infty$,
$M(\overline{w})=M(w)$, and
\begin{align}
\int_{\mathbb{C}}M(w)|dw|=1.
\end{align}
The functions $M$ and $\widetilde{M}$ are Fourier duals of each other.
\end{prop}

This is an analogue of \cite[Proposition 3.5]{IM11}, and the proof is exactly the same.

\section{Completion of the proof}

Now we finish the proof of our main Theorem \ref{thm-main}.
Among the statement of Theorem \ref{thm-main}, the properties of $M(w)$ is already
shown in the above Proposition \ref{prop-6}.   Therefore the only remaining task is to
prove \eqref{main-formula}.

First consider the case when $\Phi$ is continuous.
We have already shown the ``finite-truncation'' version of \eqref{main-formula}
as \eqref{finite-main-formula}.   We will prove that it is possible to take the limit
$N\to\infty$ on the both sides of \eqref{finite-main-formula}.

From \eqref{def_S_N} we see that the image of the mapping $S_N$ is included in the disc
of radius $\sum_{m=1}^{\infty}c_m$ for any $N$.     Therefore by 
Proposition \ref{prop_support}
we find that the support of $M_N$ for any $N$ is also included in the same disc, 
hence is the support of $M$.     The image of $f$ is clearly also bounded.
Therefore, to prove \eqref{main-formula}, we may assume that $\Phi$ is compactly
supported, hence is uniformly continuous.

Then, as $N\to\infty$, $\Phi(f_N(\alpha))$ tends to $\Phi(f(\alpha))$ uniformly in
$\alpha$.
Also, $M_N(w)\Phi(w)$ tends to $M(w)\Phi(w)$ uniformly in $w$, because of
Proposition \ref{prop-6}.
This yields that, when we take the limit $N\to\infty$ on \eqref{finite-main-formula},
we may change the integration and this limit.
Therefore we obtain \eqref{main-formula} for continuous $\Phi$.

Finally, similarly to the proof of Proposition \ref{prop_char_fct}, we can deduce the
assertion in the case when $\Phi$ is a characteristic function of a compact subset
or its complement.
This completes the proof of Theorem \ref{thm-main}.

\begin{remark}
Consider the Dirichlet series
\begin{align}\label{def_psi_s}
\Psi(s,x)=\sum_{\gamma>0}\frac{x^{i\gamma}}{(1/2+i\gamma)^s(3/2+i\gamma)^s}
\end{align}
where $s\in\mathbb{C}$.   Obviously $\Psi(1,x)=\Psi(x)$.
Because of \eqref{bound_c_m}, the series \eqref{def_psi_s} is absolutely convergent
when $\Re s>1/2$.    It is easy to see that we can 
extend Theorem \ref{thm-main} to $\Psi(s,x)$
in this domain of absolute convergence.     
\end{remark}

\begin{remark}
A generalization of the theory of the Goldbach generating function to the case with
congruence conditions was first considered by R{\"u}ppel \cite{Rupp12}, and the
generalized form of $\Psi(x)$ in this case (written in terms of the zeros of
Dirichlet $L$-functions) was determined by Suzuki \cite{Suzu17}. 
(See also \cite{BhHa20} \cite{BHMS} \cite{BhRu18}.)
It is desirable to generalize our result in the present paper to Suzuki's
generalized $\Psi$.
Probably more interesting is to consider the $\chi$-analogue;
that is instead of the average with respect to $\alpha$ as in our Theorem \ref{thm-main},
consider some analogue with respect to $\chi$ (cf. \cite{Iha08}, \cite{IM11}).
\end{remark}

%
%
\                                                                               

\end{document}